\documentclass[amsfonts]{amsart}

\usepackage{amsmath,amssymb,amsthm}
\usepackage[dvipdfmx, colorlinks=true, backref=page]{hyperref}
\usepackage[all]{xy}

\numberwithin{equation}{section}

\SelectTips{eu}{12}

\newtheorem{theorem}{Theorem}[section]
\newtheorem{proposition}[theorem]{Proposition}
\newtheorem{lemma}[theorem]{Lemma}
\newtheorem{corollary}[theorem]{Corollary}

\theoremstyle{definition}
\newtheorem{definition}[theorem]{Definition}
\newtheorem{example}[theorem]{Example}

\theoremstyle{remark}

\newcommand{\Z}{\mathbb{Z}}

\newcommand{\R}{\mathbb{R}}

\newcommand{\Conf}{\operatorname{Conf}}
\newcommand{\ind}{\operatorname{ind}}
\newcommand{\colim}{\operatorname{colim}}
\newcommand{\hocolim}{\operatorname{hocolim}}
\renewcommand{\ind}{\operatorname{ind}}

\title{van Kampen-Flores theorem for cell complexes}

\author{Daisuke Kishimoto}
\address{Faculty of Mathematics, Kyushu University, Fukuoka 819-0395, Japan}
\email{kishimoto@math.kyushu-u.ac.jp}

\author{Takahiro Matsushita}
\address{Department of Mathematical Sciences, University of the Ryukyus Nishihara-cho, Okinawa 903-0213, Japan}
\email{mtst@sci.u-ryukyu.ac.jp}



\begin{document}

  \maketitle





  \begin{abstract}
    The van Kampen-Flores theorem states that the $n$-skeleton of a $(2n+2)$-simplex does not embed into $\R^{2n}$. We give two proofs for its generalization to a continuous map from a skeleton of a certain regular CW complex (e.g. a simplicial sphere) into a Euclidean space. We will also generalize Frick and Harrison's result on the chirality of embeddings of the $n$-skeleton of a $(2n+2)$-simplex into $\R^{2n+1}$.
  \end{abstract}



  \section{Introduction}

  Recall that the van Kampen-Flores theorem \cite{F, K} states that the $n$-skeleton of a $(2n+2)$-simplex does not embed into $\R^{2n}$. Later, Volovikov generalized this theorem as follows: for any continuous map $f\colon\Delta^{r(n+2)-2}_n\to\R^d$, there are pairwise disjoint faces $\sigma_1,\ldots,\sigma_r$ of $\Delta^{r(n+2)-2}_n$ such that $f(\sigma_1)\cap\cdots\cap f(\sigma_r)\ne\emptyset$ whenever $(r-1)d\le rn$ and $r$ is a prime power \cite{V} (see also \cite{BZ, S}), where $X_n$ denotes the $n$-skeleton of a CW complex $X$.

  As in \cite{BFZ, BZ}, the generalized van Kampen-Flores theorem can be deduced from the topological Tverberg theorem by the so-called constraint method. Recently, the topological Tverberg theorem was generalized in \cite{HKTT} to a continuous map from a certain CW complex into a Euclidean space. We recall the definition of such a CW complex. To this end, we set terminology. Let $X$ be a regular CW complex. A \emph{face} of $X$ will mean its closed cell. For faces $\sigma_1,\ldots,\sigma_r$ of $X$, let $X(\sigma_1,\ldots,\sigma_r)$ denote the subspace of $X$ consisting of faces $\tau$ such that $\sigma_i\cap\tau=\emptyset$ for all $i$. Clearly, $X(\sigma_1,\ldots,\sigma_r)$ is a subcomplex of $X$. We say that $X$ is \emph{$n$-acyclic over $A$} if $\widetilde{H}_*(X;A)=0$ for $*\le n$, where $A$ is an abelian group and $(-1)$-acyclic over $A$ means non-empty.

  \begin{definition}[see \cite{HKTT}]
    Let $X$ be a regular CW complex. We say that $X$ is \emph{$k$-complementary $n$-acyclic over $A$} if $X(\sigma_1,\ldots,\sigma_i)$ is non-empty and $(n-\dim\sigma_1-\cdots-\dim\sigma_i)$-acyclic over $A$ for any pairwise disjoint faces $\sigma_1,\ldots,\sigma_i$ of $X$ with $\dim\sigma_1+\cdots+\dim\sigma_i\le n+1$ where $i=0,1,\ldots,k$.
  \end{definition}

  We give examples of $k$-complementary $n$-acyclic complexes.

  \begin{example}
    \label{simplex}
    An $n$-simplex is an $(r-1)$-complementary $(n-r)$-acyclic complex over any abelian group.
  \end{example}

  \begin{example}
    \label{simplicial sphere}
    It is proved in \cite[Proposition 2.3]{HKTT} that every simplicial $d$-sphere is $k$-complementary $(d-k)$-acyclic over any abelian group for any $k=1,2,\ldots,d+1$. Recall that a mod $p$ homology $d$-sphere is a closed manifold having the same homology as a $d$-sphere. We can easily see that the proof of \cite[Proposition 2.3]{HKTT} works verbatim for simplicial mod $p$ homology spheres because the Alexander duality holds for a mod $p$ homology sphere, where we work with mod $p$ (co)homology. Then we get that every simplicial mod $p$ homology $d$-sphere is $k$-complementary $(d-k)$-acyclic over $\Z/p$ for any $k=1,2,\ldots,d+1$.
  \end{example}

  As mentioned above, the topological Tverberg theorem is generalized in \cite{HKTT} to a continuous map from a regular CW complex which is $(r-1)$-complementary $((r-1)d-1)$-acyclic over $\Z$ into $\R^d$, where $r$ is a prime power. In this paper, we will prove the following generalization of the van Kampen-Flores theorem. We say that a regular CW complex $X$ is \emph{almost $r$-embeddable} into a space $Y$ if there is a continuous map $f\colon X\to Y$ such that for all pairwise disjoint faces $\sigma_1,\ldots,\sigma_r$ of $K$, we have $f(\sigma_1)\cap\cdots\cap f(\sigma_r)=\emptyset$. In particular, if a regular CW complex is not almost 2-embeddable into $Y$, then it is not embeddable into $Y$. Note that Volovikov's result \cite{V} mentioned above can be restated such that $\Delta^{r(n+2)-2}_n$ is not almost $r$-embeddable into $\R^d$ whenever $(r-1)d\le rn$ and $r$ is a prime power. Now we state the main theorem.

  \begin{theorem}
    \label{main}
    Let $X$ be a regular CW complex which is $(r-1)$-complementary $(r(n+1)-2)$-acyclic over $\Z/p$, where $r=p^k$ for a prime number $p$. If $(r-1)d\le rn$, then $X_n$ is not almost $r$-embeddable into $\R^d$.
  \end{theorem}

  By Example \ref{simplex}, $\Delta^{r(n+2)-2}$ is $(r-1)$-complementary $(r(n+1)-2)$-acyclic over any abelian group. So Theorem \ref{main} recovers Volovikov's generalized van Kampen-Flores theorem stated above. By Example \ref{simplicial sphere}, Theorem \ref{main} specializes to the following corollaries which are of particular interest.

  \begin{corollary}
    The $n$-skeleton of any simplicial mod $p$ homology $(r(n+2)-3)$-sphere is not almost $r$-embeddable into $\R^d$ whenever $r=p^k$ and $(r-1)d\le rn$.
  \end{corollary}

  \begin{corollary}
    The $n$-skeleton of any simplicial mod $2$ homology $(2n+1)$-sphere does not embed into $\R^{2n}$.
  \end{corollary}

  We will give two proofs for Theorem \ref{main}. One proof applies the constraint method to a generalized topological Tverberg theorem in \cite{HKTT} (Section \ref{Constraint method}). The other proof is topological and includes results on a group action on the nerve of a covering (Proposition \ref{G-map nerve}) and a skeleton of a discretized configuration space (Lemma \ref{Conf weight}), which are of independent interest. Moreover, the topological proof applies to generalize the result on chirality recently obtained by Frick and Harrison \cite{FH}. An embedding $f\colon X\to\R^{n+1}$ is called \emph{chiral} if $f$ is not isotopic to $\alpha\circ f$ for each orientation reversing homeomorphism $\alpha\colon\R^{n+1}\to\R^{n+1}$. Here, an isotopy means a homotopy $h\colon X\times[0,1]\to Y$ such that $h(-,t)\colon X\to Y$ is an embedding for each $t\in[0,1]$. Since every orientation reversing homeomorphism $\alpha\colon\R^{n+1}\to\R^{n+1}$ is isotopic to a reflection, we only need to consider a specific reflection in the definition of chirality. There are other definitions of chirality, and for the background of this subject, we refer to Flapan's book \cite{Flapan} and Section 2.1 of \cite{FH}. In \cite{FH}, the following theorem is proved.

  \begin{theorem}
    \label{Frick-Harrison}
    Each embedding $\Delta^{2n+2}_n\to\R^{2n+1}$ is chiral.
  \end{theorem}

  Note that if all embeddings of $X$ into $\R^{n+1}$ are chiral, then there is no embedding of $X$ into $\R^n$ (see \cite{FH}). Then Theorem \ref{Frick-Harrison} can be thought of as a generalization of the van Kampen-Flores theorem. As mentioned above, we can apply the results in the topological proof for Theorem \ref{main} to generalize Theorem \ref{Frick-Harrison} as follows.

  \begin{theorem}
    \label{main 2}
    Let $X$ be a regular CW complex which is 1-complementary $2n$-acyclic over $\Z/2$. Then each embedding $X_n\to\R^{2n+1}$ is chiral.
  \end{theorem}

  By Example \ref{simplicial sphere}, we get:

  \begin{corollary}
    Every embedding of the $n$-skeleton of a simplicial mod $2$ homology $(2n+1)$-sphere into $\R^{2n+1}$ is chiral.
  \end{corollary}

  \subsection*{Acknowledgements}

  The authors were partly supported by JSPS KAKENHI Grant Numbers JP17K05248 and JP19K03473 (Kishimoto), and JP19K14536 (Matsushita). The authors are grateful to the referees for useful comments. This manuscript has no associated data.


  \section{Group action on a homotopy colimit}\label{Group action on a homotopy colimit}

  This section considers a group action on a homotopy colimit induced from a group action on an index category. We will consider a homotopy colimit of a functor over a poset for simplicity, but results are easily generalized to a functor over a small category.

  We begin with a group action on a simplicial complex. For a simplicial complex $K$, let $P(K)$ denote the face poset of $K$. Then the barycentric subdivision of $K$ is $\Delta(P(K))$. We record the following basic property, which we will freely use.

  \begin{lemma}
    \label{action order complex}
    For a simplicial $G$-complex $K$, $\Delta(P(K))$ is a simplicial $G$-complex such that $\Delta(P(K))^G$ is its subcomplex.
  \end{lemma}

  \begin{proof}
    It is obvious that $\Delta(P(K))$ is a simplicial $G$-complex. If a simplex $\{x_0<\cdots<x_n\}\in\Delta(P(K))$ is fixed by $G$, then each of its face $\{x_{i_0}<\cdots<x_{i_k}\}$ is fixed by $G$ too. Thus $\Delta(P(K))^G$ is a subcomplex of $\Delta(P(K))$, completing the proof.
  \end{proof}

  Hereafter, let $P$ denote a poset. We understand $P$ as a category by regarding $x\ge y$ as a unique morphism $x\to y$. Let $F\colon P\to\mathbf{Top}$ be a functor. Recall that the homotopy colimit of $F$, denoted by $\hocolim F$, is defined as the coequalizer of two maps
  \[
    \alpha,\beta\colon\coprod_{x\ge y\in P}F(x)\times\Delta(P_{\le y})\to\coprod_{x\in P}F(x)\times\Delta(P_{\le x})
  \]
  such that $\alpha$ restricts to $F(x\ge y)\times 1\colon F(x)\times\Delta(P_{\le y})\to F(y)\times\Delta(P_{\le y})$ and $\beta$ restricts to the inclusion $F(x)\times\Delta(P_{\le y})\to F(x)\times\Delta(P_{\le x})$, where $P_{\le x}=\{y\in P\mid y\le x\}$.

  We recall from \cite{JS} a condition for a $G$-action on $P$ inducing a $G$-action on $\hocolim F$.

  \begin{definition}
    We call a functor $F\colon P\to\mathbf{Top}$ a \emph{$G$-functor} if it is equipped with a natural transformation $\eta_g\colon F\circ g\to F$ for each $g\in G$ such that $\eta_{gh} = \eta_g \circ \eta_h$ for all $g,h\in G$.
  \end{definition}

  It is immediate from the definition that if $F\colon P\to\mathbf{Top}$ is a $G$-functor, then there is a natural $G$-action on $\colim F$. Now we show:

  \begin{proposition}
    \label{action hocolim}
    If $F\colon P\to\mathbf{Top}$ is a $G$-functor, then there is a natural $G$-action on $\hocolim F$ such that the projections
    \[
      \hocolim F\to\Delta(P)\quad\text{and}\quad\hocolim F\to\colim F
    \]
    are $G$-maps.
  \end{proposition}

  \begin{proof}
    Except for the statement about the projection $\hocolim F\to\colim F$, the proposition is proved in \cite[Proposition 2.4]{JS}, and it is obvious that the projection $\hocolim F\to\colim F$ is a $G$-map.
  \end{proof}

  We consider the fixed point set of $\hocolim F$ for a $G$-functor $F$.

  \begin{proposition}
    \label{subdivision}
    Let $K$ be a simplicial $G$-complex, and let $F\colon P(K)\to\mathbf{Top}$ be a $G$-functor. Then there is a $G$-functor $\widehat{F}\colon P(\Delta(P(K)))\to\mathbf{Top}$ satisfying:
    \begin{enumerate}
      \item $\colim F\cong\colim\widehat{F}$;

      \item $(\colim\widehat{F})^G=\colim\widehat{F}\vert_{P(\Delta(P(K))^G)}$;

      \item $(\hocolim\widehat{F})^G=\hocolim\widehat{F}\vert_{P(\Delta(P(K))^G)}$.
    \end{enumerate}
  \end{proposition}

  \begin{proof}
    Define a functor $\widehat{F}\colon P(\Delta(P(K)))\to\mathbf{Top}$ by
    \begin{align*}
      &\widehat{F}(\{x_0<\cdots<x_n\})=F(x_n),\\
      &\widehat{F}(\{x_0<\cdots<x_n\}\ge\{y_0<\cdots<y_k\})=F(x_n\ge y_k).
    \end{align*}
    Then it is straightforward to see it has the desired property, where (2) and (3) use Lemma \ref{action order complex}.
  \end{proof}

  \begin{corollary}
    \label{hocolim-colim}
    Let $K$ be a simplicial $G$-complex, where $G$ is a finite group, and let $F\colon P(K)\to\mathbf{Top}$ be a $G$-functor. If $F(x>y)$ is a cofibration for each $x>y$, then the projection
    \[
      \hocolim\widehat{F}\to\colim\widehat{F}
    \]
    is a $G$-homotopy equivalence, where $\widehat{F}$ is as in Proposition \ref{subdivision}.
  \end{corollary}

  \begin{proof}
    By Proposition \ref{subdivision} and \cite[Projection Lemma 1.6]{ZZ} (cf. the proof of \cite[Proposition 5.6]{KL}), the projection
    \[
      (\hocolim\widehat{F})^H\to(\colim\widehat{F})^H
    \]
    is a homotopy equivalence for any subgroup $H$ of $G$. Then the proof is done.
  \end{proof}


  \section{Proof of Theorem \ref{main}}\label{Proof of Theorem}

  This section proves Theorem \ref{main}. Recall that the \emph{discretized configuration space} $\Conf_r(X)$, or the \emph{$r$-fold deleted product}, of a regular CW complex $X$ is defined as the subcomplex of $X^r$ consisting of faces $\sigma_1\times\cdots\times \sigma_r$ such that $\sigma_1,\ldots,\sigma_r$ are pairwise disjoint faces of $X$. The symmetric group $\Sigma_r$, hence its subgroups, act on $\Conf_r(X)$ by coordinate permutation.

  In \cite{HKTT}, a homotopy decomposition of $\Conf_r(X)$ in terms of a homotopy colimit is given, and by computing the Bousfield-Kan spectral sequence for this homotopy colimit, the following is proved in the case $A=\Z$. We can easily see that the computation works for any abelian group $A$ too.

  \begin{lemma}
    \label{Conf acyc}
    If $X$ is $(r-1)$-complementary $n$-acyclic over $A$, then $\Conf_r(X)$ is $n$-acyclic over $A$.
  \end{lemma}

  For the rest of this section, let $r=p^k$ for a prime number $k$, and let $G=(\Z/p)^k$. Note that $G$ acts on $\Conf_r(X)$ since $G$ is a subgroup of $\Sigma_r$. We consider Volovikov's index of a $G$-space \cite{V}. The \emph{index} of a $G$-space $X$, denoted by $\ind(X)$, is defined as the greatest integer $n$ such that the map
  \[
    \pi^*\colon H^*(BG;\Z/p)\to H^*(X\times_GEG;\Z/p)
  \]
  is injective for $*\le n$, where $\pi\colon X\times_GEG\to BG$ denotes the projection. The index is called weight in \cite{HK,HKTT} (cf. \cite{R}) from a view of algebraic topology, but in this paper, we use the name ``index" following a suggestion of a referee. The following properties of an index are proved in \cite{V}. If $X$ and $Y$ are $G$-spaces, then we consider $X*Y$ as a $G$-space by the diagonal $G$-action.

  \begin{proposition}
    \label{weight}
    \begin{enumerate}
      \item If there is a $G$-map $X\to Y$ between $G$-spaces $X$ and $Y$, then
      \[
        \ind(X)\le\ind(Y).
      \]

      \item If a $G$-space $X$ is $n$-acyclic over $\Z/p$, then
      \[
        \ind(X)\ge n+1.
      \]

      \item Let $S$ be a finite $G$-complex such that $S^G=\emptyset$ and $H^*(S;\Z/p)\cong H^*(S^n;\Z/p)$. Then
      \[
        \ind(S^n)=n.
      \]
      \item Let $S$ be as above. Then for a $G$-space $X$,
      \[
        \ind(X*S)\le\ind(X)+n+1.
      \]
    \end{enumerate}
  \end{proposition}

  We consider a group action on the nerve of a covering. Let $K$ be a simplicial complex having a covering
  \begin{equation}
    \label{covering}
    K=K_1\cup\cdots\cup K_r.
  \end{equation}
  Suppose that a symmetric group $\Sigma_r$ acts on $K$ by permuting $K_1,\ldots,K_r$. Let $N(K)$ denote the nerve of the covering \eqref{covering}. Then $\Sigma_r$ acts on $N(K)$ too. Define a $\Sigma_r$-functor $F\colon N(K)\to\mathbf{Top}$ by $F(K_{i_1}\cap\cdots\cap K_{i_k})=K_{i_1}\cap\cdots\cap K_{i_k}$ and the canonical inclusion.

  \begin{lemma}
    \label{hocolim-colim special}
    Let $\widehat{F}\colon P(\Delta(NK))\to\mathbf{Top}$ be the $\Sigma_r$-functor constructed from a functor $F\colon NK\to\mathbf{Top}$ in Proposition \ref{subdivision}. Then the projection $\hocolim\widehat{F}\to\colim\widehat{F}$ is a $\Sigma_r$-homotopy equivalence.
  \end{lemma}

  \begin{proof}
    This lemma immediately follows from Corollary \ref{hocolim-colim}.
  \end{proof}

  \begin{proposition}
    \label{G-map nerve}
    There is a $\Sigma_r$-map
    \[
      K\to\Delta(N(K)).
    \]
  \end{proposition}

  \begin{proof}
    By Proposition \ref{action hocolim} and Lemma \ref{hocolim-colim special}, there is a string of $\Sigma_r$-map
    \[
      K=\colim F\xrightarrow{\cong}\colim\widehat{F}\to\hocolim\widehat{F}\to\Delta(P(\Delta(NK)))\xrightarrow{\cong}\Delta(N(K)).
    \]
    Thus the proof is finished.
  \end{proof}

  We apply Proopsition \ref{G-map nerve} to our situation. Let
  \[
    P=P(\Conf_r(X)_{r(n+1)-1})-P(\Conf_r(X_n))
  \]
  and let $P_i=\{\sigma_1\times\cdots\times\sigma_r\in P\mid\dim\sigma_i>n\}$ for $i=1,\ldots,r$. Then $P=P_1\cup\cdots\cup P_r$ and each $P_i$ is an upper ideal of $P$, implying
  \[
    \Delta(P)=\Delta(P_1)\cup\cdots\cup\Delta(P_r).
  \]

  \begin{lemma}
    \label{Conf G-map}
    Let $X$ be a regular CW complex which is $(r-1)$-complementary $(r(n+1)-2)$-acyclic over $\Z/p$. Then there is a $G$-map
    \[
      \Conf_r(X)_{r(n+1)-1}\to\Conf_r(X_n)*S^{r-2}
    \]
    where $G$ acts on $S^{r-2}$ such that $(S^{r-2})^G=\emptyset$.
  \end{lemma}

  \begin{proof}
    The map
    \[
      \Delta(P(\Conf_r(X)_{r(n+1)-1}))\to\Delta(P(\Conf_r(X_n)))*\Delta(P)
    \]
    sending $\sigma\in\Delta(P(\Conf_r(X)_{r(n+1)-1}))$ to $(\sigma-\sigma\cap\Delta(P))\sqcup(\sigma\cap\Delta(P))$ is a simplicial $\Sigma_r$-map. Then we get a $G$-map
    \[
      \Conf_r(X)_{r(n+1)-1}\to\Conf_r(X_n)*\Delta(P).
    \]
    Since $P_1\cap\cdots\cap P_r=\emptyset$, $\Delta(N(\Delta(P)))$ is a $\Sigma_r$-subcomplex of $\Delta(P(\partial\Delta^{r-1}))$, we get a composition of $\Sigma_r$-maps
    \[
      \Delta(P)=\colim\widehat{F}\cong\colim\widehat{F}\to\hocolim\widehat{F}\to\Delta(N\Delta(P))\to\Delta(P(\partial\Delta^{r-1}))\cong\partial\Delta^{r-1}.
    \]
    Thus we get a $\Sigma_r$-map
    \[
      \Conf_r(X)_{r(n+1)-1}\to\Conf_r(X_n)*\partial\Delta^{r-1}
    \]
    which is also a $G$-map because $G$ is a subgroup of $\Sigma_r$. Clearly, $(\partial\Delta^{r-1})^G=\emptyset$, completing the proof.
  \end{proof}

  \begin{lemma}
    \label{Conf weight}
    If $X$ is a regular CW complex which is $(r-1)$-complementary $(r(n+1)-2)$-acyclic over $\Z/p$, then
    \[
      \ind(\Conf_r(X_n))\ge rn.
    \]
  \end{lemma}

  \begin{proof}
    By Lemma \ref{Conf acyc}, $\Conf_r(X)$ is $(r(n+1)-2)$-acyclic over $\Z/p$. Then its $(r(n+1)-1)$-skeleton $\Conf_r(X)_{r(n+1)-1}$ is $(r(n+1)-2)$-acyclic over $\Z/p$ too, and so by Proposition \ref{weight},
    \[
      \ind(\Conf_r(X)_{r(n+1)-1})\ge r(n+1)-1.
    \]
    Thus by Proposition \ref{weight} and Lemma \ref{Conf G-map},
    \begin{align*}
      r(n+1)-1&\le\ind(\Conf_r(X)_{r(n+1)-1})\\
      &\le\ind(\Conf_r(X_n)*S^{r-2})\\
      &\le\ind(\Conf_r(X_n))+r-1.
    \end{align*}
    Therefore we get $\ind(\Conf_r(X_n))\ge rn$.
  \end{proof}

  Now we are ready to prove Theorem \ref{main}.

  \begin{proof}
    [Proof of Theorem \ref{main}]
    Suppose there is a map $X_n\to\R^d$ such that for each pairwise disjoint faces $\sigma_1,\ldots,\sigma_r$ of $X_n$, $f(\sigma_1)\cap\cdots\cap f(\sigma_r)=\emptyset$. Then as in \cite[The proof of Theorem 3.10]{BZ}, there is a $G$-map
    \[
      \Conf_r(X_n)\to(\R^d)^r-\Delta,\quad(x_1,\ldots,x_r)\mapsto(f(x_1),\ldots,f(x_r))
    \]
    where $\Delta=\{(x_1,\ldots,x_r)\in(\R^d)^r\mid x_1=\cdots=x_r\}$. It is obvious that $(\R^d)^r-\Delta$ is $G$-homotopy equivalent to the unit sphere $S$ of the orthogonal complement of $\Delta\subset(\R^d)^r$. So there is a $G$-map $\Conf_r(X)\to S$. Since $S^G=\emptyset$, it follows from Proposition \ref{weight} and Lemma \ref{Conf weight} that
    \[
      rn\le\ind(\Conf_r(X_n))\le\ind(S)=(r-1)d-1.
    \]
    Thus we must have $(r-1)d > rn$, completing the proof.
  \end{proof}


  \section{Chirality}\label{Chirality}

  This section proves Theorem \ref{main 2} by applying results in the previous sections. We will use the following two lemmas.

  \begin{lemma}
    \label{weight product}
     For $i=1,2$, let $X_i$ be a $\Z/2$-space with $n_i=\ind(X_i)$, and let $(\Z/2)^2$ act on $\R^{n_1+n_2}$ by
     \begin{align*}
       &(x_1,\ldots,x_{n_1},x_{n_1+1},\ldots,x_{n_1+n_2})\cdot g_i\\
       &=
       \begin{cases}
         (-x_1,\ldots,-x_{n_1},-x_{n_1+1},\ldots,-x_{n_1+n_2})&i=1\\
         (x_1,\ldots,x_{n_1},-x_{n_1+1},\ldots,-x_{n_1+n_2})&i=2
       \end{cases}
     \end{align*}
     where $g_1=(1,0),\,g_2=(0,1)\in(\Z/2)^2$. Then every $(\Z/2)^2$-map
    \[
      X_1\times X_2\to\R^{n_1+n_2}
    \]
    has a zero.
  \end{lemma}

  \begin{proof}
    Let $G=(\Z/2)^2$ and $G_i$ be the $i$-th $\Z/2$ in $G$ for $i=1,2$. Recall that the mod 2 cohomology of $BG$ is given by
    \[
      H^*(BG;\Z/2)=\Z/2[x_1,x_2],\quad|x_i|=1
    \]
    where $x_i$ corresponds to a generator of  $H^1(BG_i;\Z/2)\cong\Z/2$ for $i=1,2$. Let $K_1$ denote the kernel of the natural map $H^*(BG;\Z/2)\to H^*((X_1\times X_2)\times_GEG;\Z/2)$. Since $(X_1\times X_2)\times_GEG\simeq(X_1\times_{G_1}EG_1)\times(X_2\times_{G_2}EG_2)$, we have
    \[
      K_1=(x_1^{n_1+1},x_2^{n_2+1}).
    \]
    Let $K_2$ denote the kernel of the natural map $H^*(BG;\Z/2)\to H^*(S^{n_1+n_2-1}\times_GEG;\Z/2)$, where the $G$-action on $S^{n_1+n_2-1}$ is a restriction of that on $\R^{n_1+n_2}$. By considering the Mayer-Vietoris exact sequence for a covering $S^{n_1+n_2-1}\times_GEG=((S^{n_1-1}\times D^{n_2})\times_GEG)\cup((D^{n_1}\times S^{n_2-1})\times_GEG)$, we can see that
    \[
      (x_1+x_2)^{n_1}x_2^{n_2}\in K_2.
    \]
    Now we suppose that a $G$-map $X_1\times X_2\to\R^{n_1+n_2}$ does not have a zero. Then by normalizing, we get a $G$-map $X_1\times X_2\to S^{n_1+n_2-1}$. This implies $K_2\subset K_1$, which contradicts to the above computation. Thus every $G$-map $X_1\times X_2\to\R^{n_1+n_2}$ must have a zero, completing the proof.
  \end{proof}


  \begin{lemma} \label{lemma psi}
    \label{Psi}
    Let $X$ be a regular CW complex $X$. Then there is a continuous map $\Psi\colon\Conf_2(X)_{2n+1}\to\R^d$ such that for each $(x,y)\in\Conf_2(X)_{2n+1}$,
    \[
      \Psi(x,y)=-\Psi(y,x)\quad\text{and}\quad\Psi^{-1}(0)=\Conf_2(X_n).
    \]
  \end{lemma}

  \begin{proof}
    Let $P=P(\Conf_2(X)_{2n+1})-P(\Conf_2(X_n))$ as in the proof of Lemma \ref{Conf G-map}. Then $\Conf_2(X)_{2n+1}$ is a $\Z/2$-subspace of $\Conf_2(X_n)*\Delta(P)$ and there is a $\Z/2$-map $f\colon\Delta(P)\to S^0=\{-1,+1\}$ as in the proof of Lemma \ref{Conf G-map}. Let $r\colon\Conf_2(X_d)* \Delta(P)-\Conf_2(X_n)\to S^0$ be the natural projection, and let $q\colon\Conf_2(X_n)* \Delta(P)\to[0,1]$ be the projection onto the join parameter such that $q^{-1}(0)=\Conf_2(X_n)$ and $q^{-1}(1)=\Delta(P)$. Now we define
    \[
      \Psi\colon\Conf_2(X)_{2n+1}\to\R,\quad x\mapsto\begin{cases}0&x\in\Conf_2(X_n)\\q(x)\cdot f(r(x))&x\not\in\Conf_2(X_n).\end{cases}
    \]
    It is straightforward to check that this map satisfies the conditions in the statement.
  \end{proof}

  \begin{proof}
    [Proof of Theorem \ref{main 2}]
    Let $\mathrm{Emb}(A,B)$ denote the space of embeddings of $A$ into $B$ with compact-open topology. Let $\Z/2$ act on $S^1$ by the antipodal map and on $\mathrm{Emb}(A,\R^N)$ by the induced map from a reflection on $\R^N$. By \cite[Lemma 2.1]{FH}, every embedding $A\to\R^N$ is chiral if and only if there is no $\Z/2$-map $S^1\to\mathrm{Emb}(A,\R^N)$. Let $F\colon X\times S^1\to\R^{2n+1}$ be a $\Z/2$-map, where $\Z/2$-actions on $X\times S^1$ and $\R^{2n+1}$ are given by
    \[
      X\times S^1\to X\times S^1,\quad(x,y)\mapsto(x,-y)
    \]
    and
    \[
      \R^{2n+1}\to\R^{2n+1},\quad(x_1,\ldots,x_{2n},x_{2n+1})\mapsto(x_1,\ldots,x_{2n},-x_{2n+1}).
    \]
    Then since any $\Z/2$-map $X_n\times S^1\to\R^{2n+1}$ extends to $\Z/2$-map $X\times S^1\to\R^{2n+1}$, it suffices to prove that there is $y\in S^1$ such that $F(-,y)\vert_{X_n}\colon X_n\to\R^{2n+1}$ is not an embedding. Define a map
    \[
      \widehat{F}\colon\Conf_2(X)_{2n+1}\times S^1\to\R^{2n+2},\quad((x_1,x_2),y)\mapsto(\Psi(x_1,x_2),F(x_1,y)-F(x_2,y)),
    \]
    where $\Psi \colon \Conf_2(X)_{2n+1} \to \R$ is a function described in Lemma \ref{lemma psi}. Let $(\Z/2)^2$ act on $\R^{2n+2}$ by
    \[
      (x_0,\ldots,x_{2n},x_{2n+1})\cdot g_i=
      \begin{cases}
        (-x_0,\ldots,-x_{2n},-x_{2n+1})&i=1\\
        (x_0,\ldots,x_{2n},-x_{2n+1})&i=2
      \end{cases}
    \]
    for $(x_0,\ldots,x_{2n},x_{2n+1})\in\R^{2n+2}$, where $g_1=(1,0),\,g_2=(0,1)\in(\Z/2)^2$. Then $\widehat{F}$ is a $(\Z/2)^2$-map. If $\widehat{F}$ does not has a zero, then by normalizing, we get a $(\Z/2)^2$-map $\Conf_2(X)\times S^1\to S^{2n+1}$. By Lemma \ref{Conf acyc}, $\Conf_2(X)_{2n+1}$ is $2n$-acyclic over $\Z/2$, and so by Proposition \ref{weight}, $\ind(\Conf_2(X)_{2n+1})\ge 2n+1$. By Proposition \ref{weight}, we also have $\ind(S^1)=1$. Then by Lemma \ref{weight product}, $\widehat{F}((x_1,x_2),y)=0$ for some $(x_1,x_2)\in\Conf_2(X)_{2n+1}$ and $y\in S^1$. Namely, $\Psi(x_1,x_2)=0$ and $F(x_1,y)=F(x_2,y)$. By the definition of $\Psi$, we have $x_1,x_2\in X_n$, and therefore $F(-,y)\colon X_n\to\R^{2n+1}$ is not injective, completing the proof.
  \end{proof}


  \section{Constraint method}\label{Constraint method}

  This section gives another proof of Theorem \ref{main} by using the constraint method (see \cite{BFZ, BZ}). Let $r=p^k$ for a prime number $p$.

  \begin{lemma}
    \label{dim}
    Let $X$ be a regular CW complex which is $(r-1)$-complementary $(d(r-1)-1)$-acyclic over $\Z/p$. Then for any continuous map $f\colon X\to\R^d$, there are pairwise disjoint faces $\sigma_1,\ldots,\sigma_r$ of $X$ such that
    \[
      f(\sigma_1)\cap\cdots\cap f(\sigma_r)\ne\emptyset\quad\text{and}\quad\dim\sigma_1+\cdots+\dim\sigma_r\le d(r-1).
    \]
  \end{lemma}

  \begin{proof}
    Suppose that $f(\sigma_1) \cap \cdots \cap f(\sigma_r) = \emptyset$ for every $r$-tuple $\sigma_1, \cdots, \sigma_r$ of pairwise disjoint faces of $X$ satisfying $\dim \sigma_1 + \cdots + \dim \sigma_r \le(r-1)d$.
    Then the natural $(\Z/p)^k$-map $\Conf_r(X)\to(\R^d)^r$ restricts to a $(\Z/p)^k$-map $\Conf_r(X)_{d(r-1)}\to(\R^d)^r-\Delta$, where $\Delta$ is as in Section \ref{Proof of Theorem}. On the other hand, by Lemma \ref{Conf acyc}, $\Conf_r(X)$ is $((r-1)d-1)$-acyclic over $\Z/p$, and so $\Conf_r(X)_{d(r-1)}$ is $((r-1)d-1)$-acyclic over $\Z/p$ too. It is easy to see that the proof of \cite[Proposition 3.2]{HKTT} works in verbatim over $\Z/p$. Then we obtain that there is no $(\Z/p)^k$-map $\Conf_r(X)_{d(r-1)}\to(\R^d)^r-\Delta$, a contradiction. Thus the proof is finished.
  \end{proof}

  \begin{proof}
    [Proof of Theorem \ref{main}]
    Since $X_n$ is a subcomplex of $X$, there is a continuous map $c\colon X\to\R$ such that $c^{-1}(0)=X_n$. Consider a map
    \[
      g\colon X\to\R^{d+1},\quad x\mapsto(f(x),c(x)).
    \]
    By Lemma \ref{Conf acyc}, $\Conf_r(X)$ is $(r(n+1)-2)$-acyclic over $\Z/p$. Then since $(r-1)(d+1)-1\le r(n+1)-2$, we can apply Lemma \ref{dim} to a map $g$, so that there are pairwise disjoint faces $\sigma_1,\ldots,\sigma_r$ of $X$ such that $g(\sigma_1)\cap\cdots\cap g(\sigma_r)\ne\emptyset$ and
    \[
      \dim\sigma_1+\cdots+\dim\sigma_r\le (r-1)(d+1)\le r(n+1)-1.
    \]
    Thus for some $x_1\in\sigma_1,\ldots,x_r\in\sigma_r$, we have $g(x_1)=\cdots=g(x_r)$. By replacing $\sigma_1,\ldots,\sigma_r$ by their faces, we may assume each $x_i$ belongs to the interior of $\sigma_i$. Since the condition $\dim\sigma_1+\cdots+\dim\sigma_r\le r(n+1) - 1$ is still satisfied, $\dim\sigma_i\le n$ for some $i$, so that $g(x_1)=\cdots=g(x_r)=0$. Thus $\dim\sigma_i\le n$ for all $n$, completing the proof.
  \end{proof}


\begin{thebibliography}{99}
    \bibitem{BFZ} P.V.M. Blagojevi\'{c}, F. Frick, and G.M. Ziegler, Tverberg plus constraints, Bull. Lond. Math. Soc. {\bf 46}, (2014), 953-967.

    \bibitem{BZ} P.V.M. Blagojevi\'{c} and G.M. Ziegler, Beyond the Borsuk-Ulam theorem: The topological Tverberg story, A journey through discrete mathematics, 273-341, Springer, Cham, 2017.


    \bibitem{Flapan} E. Flapan, {\it When topology meets chemistry: a topological look at molecular chirality.} Cambridge University Press, 2000.

    \bibitem{F} A. Flores, \"Uber $n$-dimensionale Komplexe die im $R_{2n+1}$ absolut selbstverschlungen sind, Ergeb. Math. Kolloq. \textbf{6} (1932/1934), 4-7.

    \bibitem{FH} F. Frick and M. Harrison, Spaces of embeddings: Nonsingular bilinear maps, chirality, and their generalizations, Proc. Amer. Math. Soc. \textbf{150} (2022), no. 1, 423-437.

    \bibitem{HK} Y. Hara and D. Kishimoto, Note on the cohomology of finite cyclic coverings, Topology Appl. \textbf{160} (2013), no. 9, 1061-1065.

    \bibitem{HKTT} S. Hasui, D. Kishimoto, M. Takeda, and M. Tsutaya, Tverberg's theorem for cell complexes, Bull. Lond. Math. Soc. (2023). https://doi.org/10.1112/blms.12829

    \bibitem{JS} S. Jackowski and J. S\l{}omi\'{n}ska, $G$-functors, $G$-posets and homotopy decompositions of $G$-spaces, Fund. Math. \textbf{169} (2001), 249-287.

    \bibitem{K} E.R. van Kampen, Komplexe in euklidischen R\"aumen, Abh. Math. Seminar Univ. Hamburg \textbf{9} (1933), 72-78.

    \bibitem{KL} D. Kishimoto and R. Levi, Polyhedral products over finite posets, Kyoto J. Math. \textbf{62} (2022), no. 3, 615-654.


    \bibitem{R} Y.B. Rudyak, On category weight and its applications, Topology \textbf{38} (1999), no. 1, 37-55.

    \bibitem{S} K.S. Sarkaria, A generalized van Kampen-Flores theorem, Proc. AMS \textbf{111} (1991), no. 2, 559-565.

    \bibitem{V} A.Y. Volovikov, On the van Kampen-Flores theorem, Math. Notes \textbf{59} (1996), no. 5, 477-481.

    \bibitem{ZZ} G.M. Ziegler and R.T. \u{Z}ivaljevi\'{c}, Homotopy types of subspace arrangements via diagrams of spaces, Math. Ann. \textbf{295} (1993), 527-548.
  \end{thebibliography}
\end{document}